\documentclass[11pt,twoside]{article}
\usepackage{a4wide}
\usepackage{amssymb}
\usepackage{epsfig,enumerate,amsmath,amsfonts,amssymb}
\usepackage{indentfirst}
\usepackage{pstricks}
\usepackage{setspace}
\usepackage{latexsym}
\usepackage{amsthm}
\usepackage{amsmath}
\usepackage{enumerate}
\newtheorem{pro}{Proposition}

\newtheorem{defi}{Definition}[section]
\newtheorem{thm}{Theorem}[section]

\newtheorem{lem}{Lemma}[section]
\newtheorem{lemma}{Lemma}[section]
\date{}
\usepackage[all]{xy}
\catcode`\@=11
\def\theequation{\@arabic{\c@section}.\@arabic{\c@equation}}
\catcode`\@=12

\def\N{{I\!\!N}}
\begin{document}
\title {Critical growth fractional elliptic systems with exponential nonlinearity}

\author{ {\sc J.Giacomoni}\footnote{LMAP (UMR CNRS 5142) Bat. IPRA,
  Avenue de l'Universit\'e
   F-64013 Pau, France. email:jacques.giacomoni@univ-pau.fr} , {\sc Pawan Kumar Mishra}\footnote{Department of Mathematics, Indian Institute of Technology Delhi,
Hauz Khaz, New Delhi-110016, India. e-mail: pawanmishra31284@gmail.com} \; and {\sc K.Sreenadh}\footnote{Department of Mathematics, Indian Institute of Technology Delhi,
Hauz Khaz, New Delhi-110016, India.
 e-mail: sreenadh@gmail.com}
}

\maketitle
%
%
%
\begin{abstract}
\noindent We study the existence  of positive solutions for  the system of fractional elliptic equations of the type,
\begin{equation*}
\begin{array}{rl}
(-\Delta)^{\frac{1}{2}} u &=\frac{p}{p+q}\lambda f(x)|u|^{p-2}u|v|^q + h_1(u,v) e^{u^2+v^2},\;\textrm{in}\; (-1, 1),\\
(-\Delta)^{\frac{1}{2}} v &=\frac{q}{p+q}\lambda f(x)|u|^p|v|^{q-2}v + h_2(u,v) e^{u^2+v^2},\;\textrm{in}\; (-1, 1),\\
 u,v&>0 \;\textrm{in } \; (-1,1),\\
   u&=v=0 \; \text{in} \; \mathbb R\setminus (-1,1).
 \end{array}
\end{equation*}
where {$1<p+q<2$}, $h_1(u,v)=(\alpha{+}2u^2)|u|^{\alpha-2}u|v|^\beta, h_2(u,v)=(\beta{+}2v^2) |u|^\alpha |v|^{\beta-2}v$ and ${\alpha+\beta>2}$.
Here $(-\Delta)^{\frac{1}{2}}$ is the fractional Laplacian operator. We show the existence of multiple solutions for suitable range of $\lambda$ by analyzing the fibering maps and the corresponding Nehari manifold. We also study the existence of positive solutions for a superlinear system with critical growth exponential nonlinearity.
 \end{abstract}
\section{Introduction}
 \setcounter{equation}{0}
\noindent We study the following system for existence and multiplicity of solutions
\begin{equation*}
(P_{\lambda})\;
\left\{
\begin{array}{rl}
(-\Delta)^{\frac{1}{2}} u &=\frac{p}{p+q}\lambda f(x)|u|^{p-2}u|v|^q + h_1(u,v) e^{u^2+v^2},\;\textrm{in}\; (-1, 1),\\
(-\Delta)^{\frac{1}{2}} v &=\frac{q}{p+q}\lambda f(x)|v|^{q-2}v|u|^p + h_2(u,v) e^{u^2+v^2},\;\textrm{in}\; (-1, 1),\\
 u,v&>0 \;\textrm{in } \; (-1,1),\\
   u&=v=0 \; \text{in} \; \mathbb R\setminus (-1,1).
 \end{array}
 \right.
\end{equation*}
where {$1<p+q<2$}, {$h_1(u,v)=(\alpha+2u^2)|u|^{\alpha-2}u|v|^\beta \;\textrm{and}\; h_2(u,v)=(\beta +2v^2) |u|^\alpha |v|^{\beta-2}v$, $\alpha+\beta>2$,} $\lambda>0$ and $f\in L^r(-1, 1)$, for suitable choice of {$r>1$}, is sign changing. {Here $(-\Delta)^{\frac{1}{2}}$ is the $\frac{1}{2}$-Laplacian operator defined as
\begin{equation*}
 (-\Delta)^{\frac{1}{2}}u=\int_{\mathbb{R}} \frac{(u(x+y)+u(x-y)-2u(x))}{|y|^{2}}dy\;\;\;\;\;\textrm{for all}\; \; x\in \mathbb{R}.
 \end{equation*}}
%
\noindent The fractional Laplacian  operator has been a classical topic in Fourier analysis and nonlinear partial differential equations.  Fractional operators are  involved in financial mathematics, where Levy processes with jumps appear in modeling the asset prices (see \cite{app}).
Recently the semilinear equations involving the fractional Laplacian has attracted many researchers. The critical exponent problems for fractional Laplacian have been studied in \cite{sv, tan}. Among the works dealing with fractional elliptic equations with critical exponents we cite also \cite{tfwc, XcT, mb,mb1} and references there-in, with no attempt to provide a complete list.

\noindent In the local setting, the semilinear ellitpic systems involving Laplace operator with exponential nonlinearity has been investigated in \cite{DMR, nhsr}. {The} case of polynomial nonlinearities involving linear and quasilinear operators has been studied in \cite{kadie, alvc,yb2,yb3,yb4,yb5,yb6}. Furthermore, these results for sign changing nonlinearities with polynomial type subcritical and critical growth have been obtained in \cite{yb1, bruw, uwcyh,psree,wutf} using Nehari manifold and fibering map analysis. These problems for exponential growth nonlinearties is studied in \cite {DMR}:
\begin{equation*}
\left\{
\begin{array}{rl}
-\Delta u &=g(v),\;
-\Delta v =f(u),\;\textrm{in}\; \Omega\\
   u&=v=0 \; \text{on} \; \partial \Omega.
 \end{array}
 \right.
\end{equation*}
where functions $f$ and $g$ have critical growth {in the sense of} {Trudinger}-Moser inequality and have shown the existence of a nontrivial weak solutions in both sub-critical as well as critical growth case. In \cite{nhsr}, authors have considered the elliptic system with exponential growth perturbed by a concave growth term and established the global multiplicity results with respect to the parameter.

\noindent Recently semilinear equations involving fractional Laplacian with exponential nonlinearities have been studied by many authors.  Among them we  cite \cite{ionzsq, lm, yu, jpsree} and the references therein. The system of equations with fractional Laplacian operator with polynomial subcritical and critical Sobolev exponent have been studied in \cite{wcsd, sqw}. Our aim in this article is to generalize the result in \cite{jpsree} for fractional elliptic systems. In \cite{sqw}, authors considered the problem
\begin{align*}
(-\Delta )^s u & = \lambda |u|^{q-2}u + \frac{2\alpha}{\alpha+\beta}|u|^{\alpha -2}u |v|^{\beta},\; \; (-\Delta )^s   = \lambda |v|^{q-2}v + \frac{2\alpha}{\alpha+\beta}|u|^{\alpha} |v|^{\beta-2}v \; in \; \Omega,\\
u&=v=0 \; on \; \partial \Omega.
\end{align*}
where $\Omega \subset \mathbb{R}^
N$ is a smooth bounded domain, $\lambda, � > 0, 1 < q < 2$ and $\alpha > 1, \beta > 1$ satisfy
$\alpha + \beta =  2N/(N - 2s), s \in (0, 1)$ and $N > 2s.$ They studied the associated Nehari manifold using the fibering maps and show the existence of non-negative solutions arising out of structure of manifold. Our results in section 3 extends these results to the exponential case.

\noindent The variational functional $J_{\lambda}$ associated to the problem$(P_{\lambda})$ is given as
\begin{equation*}
{J_{\lambda}(u,v)=\frac{1}{2}\int_{-1}^1\left(|(-\Delta)^{\frac{1}{4}}u|^2 +|(-\Delta)^{\frac{1}{4}}v|^2\right)dx -\frac{\lambda}{p+q}\int_{-1}^1f(x)|u|^p|v|^qdx-\int_{-1}^1 G(u,v)dx.}
\end{equation*}
\begin{defi}
 $(u,v)\in H^{s}_{0}(-1,1)\times H^{s}_{0}(-1,1)$  is called weak solution of $(P_\lambda)$ if
\begin{align*}
\int_{-1}^{1} \left((-\Delta)^{\frac{1}{4}} u (-\Delta)^{\frac{1}{4}}\phi + (-\Delta)^{\frac{1}{4}} v (-\Delta)^{\frac{1}{4}}\psi\right)dx &= \lambda \int_{-1}^{1}\left( |u|^{p-2}|v|^{q}u\phi+|v|^{q-2}|u|^{p}v \psi\right) dx\\
&\quad + \int_{-1}^{1}\left(h_1(u,v)\phi +h_2(u,v)\psi\right) e^{u^2+v^2}dx
\end{align*}
for all $(\phi,\psi)\in H^{s}_{0}(-1,1)\times H^{s}_{0}(-1,1).$
\end{defi}
\noindent In the beautiful work \cite{CS}, Caffarelli and Silvestre used Dirichlet-Neumann maps to transform the non-local equations involving fractional Laplacian into a local problem. This approach attracted lot of interest by many authors recently to address the existence and multiplicity of solutions using variational methods.
In \cite{CS}, it was shown that for any $v\in H^{\frac{1}{2}}(\mathbb{R}), $ the unique function $w(x,y)$ that minimizes the weighted integral
\[\mathcal{E}_{\frac{1}{2}}(w)=\int_{0}^{\infty}\int_{\mathbb{R}} |\nabla w(x,y)|^2 dx dy\]
over the set $\left\{w(x,y): \mathcal{E}_{\frac{1}{2}}(w)<\infty, w(x,0)=v(x)\right\}$ satisfies
{$\int_{\mathbb{R}}|(-\Delta)^{\frac{1}{4}}v|^2 = \mathcal{E}_{\frac{1}{2}}(w).$} Moreover $w(x,y)$ solves the boundary value problem
\begin{equation*}
-\text{div}(\nabla w)=0 \; \text{in}\; \mathbb{R}\times \mathbb{R}_+,\;\; w(x,0)=v(x)\quad
\frac{\partial w}{\partial \nu}=(-\Delta)^{1/2} v(x)  \end{equation*}
where $\frac{\partial w}{\partial \nu}=\displaystyle\lim_{y\rightarrow 0^+}\frac{\partial w}{\partial y}(x, y)$. In case of bounded domains, in \cite{XcT}, it was observed that the harmonic extension problem  is
{$$ \quad (P_E)\; \left\{
\begin{array}{rrll}
 \quad  -\Delta w &= &0, \;w>0\; \; \text{in}\;\mathcal{C}=(-1,1)\times (0,\infty),\\
  w &=&0 \; \text{on}\; \{-1,1\}\times (0,\infty),\\
  \frac{\partial w}{\partial \nu}&=&(-\Delta)^{1/2} v \; \text{on} \; (-1,1)\times \{0\}.\\
\end{array}
\right.
$$}
This can be solved on the space $H^{1}_{0,L}(\mathcal{C})$, which is defined as
\[ H^{1}_{0,L}(\mathcal{C})=\{v\in H^1(\mathcal{C}): v=0 \; \text{a.e. in }\; \{-1,1\}\times (0,\infty)\}\]
equipped with the norm $\displaystyle \|w\|= \left(\int_{\mathcal{C}} |\nabla w|^2 dx dy \right)^{\frac{1}{2}}.$
The space $H^{\frac{1}{2}}(\mathbb R)$ is the Hilbert space with the norm defined as
\begin{equation*}
\|u\|_{H}^{2}= \|u\|_{L^2({\mathbb R})}+\int_{\mathbb R}|(-\Delta)^{\frac{1}{4}}u|^2dx.
\end{equation*}
The space $H_0^{\frac{1}{2}}({\mathbb R})$ is the completion of $C_0^{\infty}({\mathbb R})$ under
$[u]=\left(\int_\mathbb R|(-\Delta)^\frac{1}{4}u|^2dx\right)^{\frac{1}{2}}$. In case of bounded intervals, say $I=(-1,1)$, the function spaces are defined as
\[X=\{ u\in H^{\frac{1}{2}}(\mathbb R): \; u=0 \; \text{in}\; \mathbb R\backslash (-1,1)\}\]
equipped with the norm
\begin{equation*}
\displaystyle \|u\|_X=[u]_{H^{1/2}(\mathbb R)}={\sqrt{2\pi}} \|(-\Delta)^{1/4}u\|_{L^2(\mathbb R)}.
\end{equation*}

\noindent Using the above idea, in the case of systems we have the following extension problem
 {\begin{equation*}
(P_E)_{\lambda}\;
\left\{\begin{array}{rl}
 -div(\nabla w_1)=0,&-div(\nabla w_2)=0 \quad\textrm{in}\; \mathcal{C}=(-1,1)\times (0,\infty),\\
 w_1,w_2&>0 \quad\quad\quad\quad\quad\quad\textrm{in}\; \mathcal{C}=(-1,1)\times (0,\infty),\\
 w_1=w_2 &=0 \; \text{on}\; \{-1,1\}\times (0,\infty),\\
  \frac{\partial w_1}{\partial \nu}&= \frac{p}{p+q}\lambda f(x)|u|^{p-2}u|v|^q+ h_1(u,v)e^{u^2+v^2}\;\; \textrm{on}\;\; (-1,1)\times \{0\},\\
  \frac{\partial w_2}{\partial \nu}&= \frac{q}{p+q}\lambda f(x)|v|^{q-2}v|u|^p+ h_2(u,v)e^{u^2+v^2}\;\; \textrm{on}\;\; (-1,1)\times \{0\}
\end{array}
\right.
\end{equation*}}
where $\frac{\partial w_1}{\partial \nu}=\displaystyle\lim_{y\rightarrow 0^+}\frac{\partial w_1}{\partial y}(x, y)$ and $\frac{\partial w_2}{\partial \nu}=\displaystyle\lim_{y\rightarrow 0^+}\frac{\partial w_2}{\partial y}(x, y)$.\\
The problem $(P_E)_{\lambda}$ can be solved on the space {$\mathcal W^{1}_{0,L}(\mathcal{C})$, which is defined as
\[ \mathcal W^{1}_{0,L}(\mathcal{C})=\{(w_1,w_2)\in H^1_{0,L}(\mathcal{C})\times H^1_{0,L}(\mathcal{C}) : w_1=w_2=0 \; \text{a.e. in }\; \{-1,1\}\times (0,\infty)\}\]}
equipped with the norm
\[\|w\|= \left(\int_{\mathcal{C}} |\nabla w_1|^2 dx dy+\int_{\mathcal{C}} |\nabla w_2|^2 dx dy \right)^{\frac{1}{2}}.\]
\noindent The Moser-Trudinger trace inequality for $\mathcal W^{1}_{0,L}(\mathcal{C})$ follows from Martinazzi \cite{lm} which improves the former results of Ozawa \cite{ozaw} and Kozono, Sato \& Wadade \cite{ksw} and uses Lemma 1.1 in Megrez, Sreenadh \& Khalidi \cite{nhsr} which adapts the Trudinger-Moser Inequality for systems:
\begin{thm}\label{mtsys}
For any $\alpha \in (0, \pi]$ there exists $C_\alpha>0$  such that
\begin{equation}\label{eq12}
\displaystyle \sup_{\|w\|\leq1} \int_{-1}^1 e^{\alpha (w_1(x, 0)^2+w_2(x,0)^2)}dx\leq C_\alpha.
\end{equation}
\end{thm}
\noindent The variational functional $I_\lambda:\mathcal W^1_{0, L}(\mathcal C) \rightarrow \mathbb R$ associated to $(P_E)_{\lambda}$ is defined as,
\[I_\lambda(w)=\frac{1}{2}\int_\mathcal C \left(|\nabla w_1|^2  + |\nabla w_2|^2\right) dxdy
 -\frac{\lambda}{p+q} \int_{-1}^1 (f(x)|w_1(x, 0)|^p|w_2(x,0)|^q - G(w(x,0)))dx
\]
{where $w=(w_1,w_2)$}.
Any $w=(w_1,w_2)\in \mathcal W^1_{0, L}(\mathcal C)$ is called the weak solution of the problem $(P_E)_{\lambda}$ if for any $\phi=(\phi_1,\phi_2) \in \mathcal W^1_{0, L}(\mathcal C)$
\begin{align*}
\int_\mathcal C &\left(\nabla w_1 \nabla \phi_1d +\nabla w_2\nabla \phi_2 \right)dxdy =\lambda \frac{p}{p+q}\int_{-1}^1f(x)|w_1(x, 0)|^{p-2}w_1(x, 0)|w_2(x,0)|^q\phi_1(x, 0)dx\\
&+\lambda \frac{q}{p+q}\int_{-1}^1f(x)|w_1(x,0)|^p|w_2(x, 0)|^{q-2}w_2(x, 0)\phi_2(x, 0) dx
\\&+\int_{-1}^1 g_1(w_1(x, 0),w_2(x,0))\phi_1(x, 0)dx +\int_{-1}^1g_2(w_1(x, 0),w_2(x,0))\phi_2(x, 0) dx.
\end{align*}
It is clear that critical points of $I_\lambda$ in $\mathcal W^1_{0, L}(\mathcal C)$ corresponds to the critical points of $J_\lambda$ in $X\times X$.
So as noted in the introduction, we will look for solutions  $w=(w_1,w_2)$ of $(P_E)_{\lambda}$. \\

\noindent In this paper we establish the multiplicity results for fractional systems with critical exponential nonlinearities. In section 2 we derive some {delicate technical} estimates to study the structure of Nehari manifold using the associate fibering maps. In section 3 we show the existence of two solutions that arise out of minimizing the functional over the  Nehari manifold. Here we show the existence of second solution using mountain-pass arguments over positive cones and compactness analysis of Palais-Smale sequences. We show the following multiplicity result.
\begin{thm}\label{mht1}
There exists a $\Lambda_0>0$ such that  $(P_{\lambda})$ has at least two solutions for every $\lambda \in (0, \Lambda)$.
\end{thm}
\noindent {Similarly we can deal with the superlinear case, i.e.}
\begin{equation*}
(P)\left\{
\begin{array}{rl}
\quad \; (-\Delta)^{1/2}u &= h_1(u,v),\;\textrm{in}\; (-1,1),\\
(-\Delta)^{1/2}v &= h_2(u,v),\;\textrm{in}\; (-1,1),\\
u,v&>0  \; \text{in} (-1,1),\\
u=v&=0 \; \text{in} \;\mathbb{R}\backslash (-1,1)
\end{array}
 \right.
\end{equation*}
where $h_1(u,v)$ and $h_2(u,v)$ satisfy the following conditions:
\begin{enumerate}
\item[(h1)] $h_1,h_2\in C^1(\mathbb R^2), h_1(s,t)=h_2(s,t)=0 $ for $s\le 0, t\le 0$, $h_1(s,t)=\frac{\partial H}{\partial s}(s,t)$, $h_2(s,t)=\frac{\partial H}{\partial t}(s,t)>0$ for $s>0, t>0$  and for any $\epsilon>0$,
$\displaystyle\lim_{(s,t)\to\infty}h_1(s,t)e^{-(1+\epsilon)(s^2+t^2)}=0$ and $\displaystyle\lim_{(s,t)\to\infty}h_2(s,t)e^{-(1+\epsilon)(s^2+t^2)}=0.$
\item[(h2)] There exists $\mu>2$ such that for all $s,t>0$,
\[0\le \mu H(s,t) \le s h_1(s,t)+t h_2(s,t).\]
\item[(h3)] There exist positive  constants $s_0>0, t_0>0$ and $M_1>0, M_2>0$ such that\\
$ H(s,t)\le M_1 h_1(s,t)+ M_2 h_2(s,t)\;\mbox{for all}\; s>s_0$ and $t>t_0$.
\item[(h4)] $\displaystyle \lim_{s,t\rightarrow \infty} \left( s h_1(s,t) + t h_2(s,t) \right)e^{-(s^2+t^2)}=\infty.$
\item[(h5)]$\displaystyle \limsup_{s,t \rightarrow 0}\frac{2 H(s,t)}{s^2+t^2}<\lambda_1,$
where $\displaystyle {\lambda_1=2\displaystyle \min_{w\in \mathcal H^1_{0, L}(\mathcal C)}\left\{\int_C|\nabla w_1|^2 \; ; \;\int_{-1}^1|w_1(x,0)|^2dx=1\right\}.}$
\end{enumerate}
An example of a function satisfying above assumptions is
{$$
H(s,t)= \begin{cases}
(s^{4}+t^4)e^{s^{2}+t^{2}} & \text{if } s>0, \; t>0,\\
 0 & \text{otherwise}.
 \end{cases}
$$}
The variational functional associated to the problem $(P)$ is given as
\begin{equation*}
J(u,v)=\frac{1}{2}\int_{-1}^1|(-\Delta)^{\frac{1}{4}}u|^2dx+\frac{1}{2}\int_{-1}^1|(-\Delta)^{\frac{1}{4}}v|^2dx -\int_{-1}^1 H(u,v) dx.
\end{equation*}
Following very closely the variational approach in Subsection \ref{3.1}, we {prove in Section 4}:
\begin{thm}\label{mht0}
Suppose $(h1)-(h5)$ are satisfied. Then the problem $(P)$ has a solution.
\end{thm}
\section{Nehari manifold and fibering maps}
\noindent {In this section}, we consider the Nehari manifold associated to the problem $(P_E)_{\lambda}$ as
\begin{equation*}
\mathcal{N}_{\lambda}=\{w\in \mathcal W^1_{0, L}(\mathcal C)\;|\; \langle I_{\lambda}^{\prime}(w),w\rangle=0\}.
\end{equation*}
Thus $w=(w_1,w_2)\in \mathcal{N}_{\lambda}$ if and only if
\begin{align}\label{nlam}\nonumber
\|w\|^2=\lambda\int_{-1}^1 f(x)|w_1(x, 0)|^{p}|w_2(x,0)|^q dx
&+\int_{-1}^1 g_1(w_1(x, 0),w_2(x,0))w_1(x, 0)dx\\&+\int_{-1}^1 g_2(w_1(x, 0),w_2(x,0))w_2(x, 0)dx.
\end{align}
Now for every $w=(w_1,w_2)\in \mathcal W^1_{0, L}(\mathcal C)$, we define {the} fiber map $\Phi_w: \mathbb{R}^+\rightarrow \mathbb{R}$ as $\Phi_w(t)=I_\lambda(tw)$. Thus $tw\in \mathcal{N}_\lambda$ if and only if
\begin{align*}
\Phi_w^{\prime}(t)=t\|w\|^2 &-\lambda t^{p+q-1}\int_{-1}^1f(x)|w_1(x, 0)|^p|w_2(x,0)|^q dx\\
&-\int_{-1}^1 g_1(tw(x, 0))w_1(x, 0) dx
-\int_{-1}^1 g_2(tw(x, 0))w_2(x, 0) dx=0.
\end{align*}
In particular, $w\in \mathcal{N}_{\lambda}$ if and only if
\begin{align*}
\|w\|^2 -\lambda \int_{-1}^1f(x)|w_1(x, 0)|^p|w_2(x,0)|^q dx
&-\int_{-1}^1 g_1(w(x, 0))w_1(x, 0) dx\\
&-\int_{-1}^1 g_2(w(x, 0))w_2(x, 0) dx=0.
\end{align*}
Also

\begin{align}\nonumber
&\Phi^{\prime\prime}(w)(1)=\|w\|^2 -(p+q-1)\lambda\int_{-1}^1f(x)|w_1(x, 0)|^p|w_2(x,0)|^q dx \\\label{fi2}
&-\int_{-1}^1 \left\{(\alpha+\beta+2w_1^2+2w_2^2)(\alpha+\beta-1+2w_1^2+2w_2^2)+4(w_1^2+w_2^2)\right\}|w_1|^\alpha|w_2|^\beta e^{w_1^2+w_2^2} dx.
\end{align}
We split $\mathcal{N}_{\lambda}$ into three parts as
\begin{equation*}
\mathcal{N}_{\lambda}^{\pm}=\{w\in \mathcal{N}_{\lambda}\;|\;\Phi^{\prime\prime}(w)(1)\gtrless 0\} \; \text{and}\; \mathcal{N}_{\lambda}^0=\{w\in \mathcal{N}_{\lambda}\;|\;\Phi^{\prime\prime}(w)(1)=0\}.\\
\end{equation*}

In order to prove Theorem \ref{mht1}, we need following version of Lemma 3.1 of \cite{psree}.
\begin{lem}\label{gmal} Let $\Gamma\subset \mathcal W^1_{0,L}(\mathcal C)$ such that for any $w\in \Gamma$,\\
$\|w\|^2\leq \frac{1}{2-p-q}\int_{-1}^1 (\alpha+\beta-{(p+q)}+2w_1^2+2w_2^2)(\alpha+\beta+2+2w_1^2+2w_2^2)|w_1|^\alpha|w_2|^\beta e^{w^2_1+w^2_2}dx$. Then there exists a positive $\Lambda_0$ such that
\begin{align}\label{gma}\nonumber
\Gamma_0:=\displaystyle \inf_{w\in \Gamma\setminus\{0\}}\left\{ \displaystyle \int_{-1}^1 (\alpha +\beta-{2}+2w_1^2+2w_2^2)(\alpha+\beta+2w_1^2+2w_2^2)|w_1|^\alpha|w_2|^\beta e^{w^2_1+w^2_2}dx\right.\\-\left.(2-p-q)\lambda \int_{-1}^1f(x)|w_1(x, 0)|^p|w_2(x,0)|^q dx \right \}>0,
\end{align}
for every $\lambda\in (0,\Lambda_0)$.
\end{lem}
\begin{proof}
We divide the proof into three steps:\\
\textbf{Step 1:} $\inf_{w\in \Gamma\setminus\{0\}}\|w\|>0$. \\
Suppose not. Then there exists $\{w_k\}$ in $\mathcal W^1_{0, L}(\mathcal C)\setminus \{0\}$ such that $\|w_k\|\rightarrow 0$ and
\begin{equation}\label{us1}
\|w\|^2\leq \frac{1}{2-p-q}\int_{-1}^1 (\alpha+\beta-p-q+2w_1^2+2w_2^2)(\alpha+\beta+2+2w_1^2+2w_2^2)|w_1|^\alpha|w_2|^\beta e^{w^2_1+w^2_2}dx. \end{equation}
 Now, using {Theorem} \ref{mtsys}, we get
\begin{align*}
\int_{-1}^{1} (\alpha+\beta-p-q&+2w_{k_1}^2+2w_{k_2}^2)(\alpha+\beta+2+2w_{k_1}^2+2w_{k_2}^2)|w_{k_1}|^\alpha|w_{k_2}|^\beta e^{w^2_{k_1}+w^2_{k_2}}dx \\
&=(\alpha+\beta-p-q)(\alpha+\beta+2)\int_{-1}^1|w_{k_1}|^\alpha|w_{k_2}|^\beta e^{w^2_{k_1}+w^2_{k_2}}dx\\
&+2(2\alpha+2\beta+2-p-q)\int_{-1}^1(w_{k_1}^2+w_{k_2}^2) |w_{k_1}|^{\alpha}|w_{k_2}|^\beta e^{w^2_{k_1}+w^2_{k_2}}dx\\
&+4\int_{-1}^1 (w_{k_1}^2+w_{k_2}^2)^2|w_{k_1}|^\alpha|w_{k_2}|^{\beta} e^{w^2_{k_1}+w^2_{k_2}}dx\\
&\leq C_1\|w_{k}\|^{\alpha+\beta}+C_2\|w_{k}\|^{\alpha+\beta+2}+C_3\|w_{k}\|^{\alpha+\beta+4}.
\end{align*}
Hence from equation \eqref{us1} and the last inequality, we get
$\|w_k\|^{\alpha+\beta}\geq C>0$, which is a contradiction as $\alpha+\beta>2$. Hence the claim is proved.\\
\textbf{Step 2:} Let $C_*=\inf_{w\in\Gamma\setminus\{0\}}\displaystyle \int_{-1}^1 (\alpha +\beta{-2}+2w_1^2+2w_2^2)(\alpha+\beta+2w_1^2+2w_2^2)|w_1|^\alpha|w_2|^\beta e^{w^2_1+w^2_2}dx$. Then $C_*>0$.\\
Proof follows from step 1 and the definition of $\Gamma$.\\
\textbf{Step 3:} Let $\lambda\in (0,\Lambda_0) \;\textrm{for}\; \Lambda_0=\frac{C_*^{1-a}}{2-p-q}$. Then  equation \eqref{gma} holds. {Indeed, let $\alpha_0=\frac{pr}{r-1}>1$ and $\beta_0=\frac{qr}{r-1}>1$ and $a=1-\frac{1}{r}$. Then we have:}
 {\begin{align*}
\lambda \int_{-1}^1f(x)&|w_1(x, 0)|^p |w_2(x,0)|^q dx \leq\lambda C_{f}\left(\int_{-1}^1|w_1(x, 0)|^{\alpha_0}|w_2(x,0)|^{\beta_0} dx\right)^a\\
&\leq\frac{(\lambda C_{f})}{C_*^{1-a}}\displaystyle \int_{-1}^1 (\alpha +\beta{-2}+2w_1^2+2w_2^2)(\alpha+\beta+2w_1^2+2w_2^2)|w_1|^\alpha|w_2|^\beta e^{w^2_1+w^2_2}dx
\end{align*}}
where $C_f=\|f\|_{L^\frac{1}{1-a}((-1, 1))}$. Thus if $(\lambda C_{f})<\Lambda_0=\frac{C_*^{1-a}}{2-p-q}$. Then equation \eqref{gma} holds.
\end{proof}
\noindent Now we discuss the behavior of $\Phi_w$ with respect to  $\int_{-1}^1f(x)|w_1(x, 0)|^p|w_2(x,0)|^q dx $.\\
\textbf{Case 1:} $  \int_{-1}^1f(x)|w_1(x, 0)|^p|w_2(x,0)|^q dx\leq 0$. \\
We define
\begin{equation*}
\Psi_w(t)=t^{2-p-q}\|w\|^2-t^{1-p-q}\int_{-1}^1g_1(tw(x,0))w_1(x,0)dx-t^{1-p-q}\int_{-1}^1g_2(tw(x,0))w_2(x,0)dx.
 \end{equation*}
 It is clear that {$\Psi_w(0)=0$} and $tw\in \mathcal{N}_{\lambda}$ if and only if \\
 $\Psi_w(t)=\lambda\int_{-1}^1f(x)|w_1(x, 0)|^p|w_2(x,0)|^q dx$.
Observe that $\displaystyle \lim_{t\rightarrow\infty}\Psi_w(t)\rightarrow -\infty$, $\displaystyle \lim_{t\rightarrow\infty}\Psi'_w(t)\rightarrow -\infty$, {$\Psi''_w(t)\leq 0$ for any $t>0$} and $\displaystyle \lim_{t\rightarrow 0^+}\Psi'_w(t)>0.$ Hence there exists a unique $t_*(w)>0$ such that $\Psi_w(t)$ is increasing in $(0, t_*)$, decreasing in $(t_*, \infty)$. Hence for all values of $\lambda$ there exists a unique $t^-(w)>t_*(w)$ such that $\Psi_w(t^-)=\lambda\int_{-1}^1f(x)|w_1(x, 0)|^p|w_2(x,0)|^q dx\leq 0$. Since $\Psi'_w(t)<0$ for $t>t_*$, using that the relation $\Phi_{tw}^{\prime\prime}(1)=t^{p+q+1}\Psi'_w(t)$ {is valid for $t=t^-$, we get} $ t^-w\in \mathcal{N}_{\lambda}^-$.\\
\textbf{Case 2:}$\int_{-1}^1f(x)|w_1(x, 0)|^p|w_2(x,0)|^q dx >0$.\\
As discussed in case 1, we have $\Psi'_w(t_*)=0$ which implies $t_*w\in\Gamma\setminus\{0\}$. Now
\begin{align*}\nonumber
\Psi_w(t_*)\geq\frac{1}{(2-p-q)t_*^{p+q}}\int_{-1}^1(\alpha+\beta-{2}+2|t_*w_1|^2&+2|t_*w_2|^2)(\alpha+\beta+2|t_*w_1|^2+2|t_*w_2|^2)\\
&\times|t_*w_1|^\alpha|t_*w_2|^\beta e^{|t_*w_1|^2+|t_*w_2|^2}dx.
\end{align*}
Now from Lemma \ref{gmal}, $\Psi_w(t_*)-\lambda\int_{-1}^1f(x)|w_1(x,0)|^p|w_2(x,0)|^qdx >0$.
Therefore there exists unique $t^+(w)<t_*(w)<t^-(w)$ such that $\Psi'_w(t^+)>0$, $\Psi'_w(t^-)<0$, and $\Psi_w(t^+)=\Psi_w(t^-)=\lambda\int_{-1}^1f(x)|w_1(x,0)|^p|w_2(x,0)|^qdx$. As a consequence $t^+w\in \mathcal{N}^+_{\lambda}$ and $t^-w\in \mathcal{N}^-_{\lambda}$. Moreover since $\Phi_w'(t^-)=\Phi_w'(t^+)=0$, $\Phi_w'(t)<0$ on $(0,t^+)$ and $\Phi_w'(t)>0$ on $(t^+,t^-)$, $I_\lambda(t^+w)=\displaystyle \min_{0<t<t^-}I_\lambda(tw)$ and $I_\lambda(t^-w)=\displaystyle \max_{t>t_*}I_\lambda(tw)$. From above discussion we have the following lemma.
\begin{lem}\label{uhm}
For any $w\in \mathcal W^1_{0, L}(\mathcal C)$,
\begin{enumerate}[(i)]
\item If $\int_{-1}^1f(x)|w_1(x,0)|^p|w_2(x,0)|^qdx\leq 0$, there exists a unique $t^-(w)>0$ such that $t^-w\in \mathcal{N}_{\lambda}^-$ for every $\lambda>0$.
\item If $\int_{-1}^1f(x)|w_1(x,0)|^p|w_2(x,0)|^qdx>0$, there exists unique $t^+(w)<t_*(w)<t^-(w)$ such that $t^+w\in \mathcal{N}^+_{\lambda}$ and $t^-w\in \mathcal{N}^-_{\lambda}$ for every $\lambda_0$ satisfying Lemma \ref{gmal}. Moreover $I_\lambda(t^+ w)<I_\lambda(tw)$ for any $t\in [0,t^-]$ such that $t\neq t^+$ and $I_\lambda(t^-w)=\displaystyle \max_{t>t_*}I_\lambda(tw)$ .
    \end{enumerate}
\end{lem}
Now, the next Lemma shows that $\mathcal{N}_\lambda$ is a manifold.
\begin{lem}
Let $\Lambda_0$ as in Lemma \ref{gmal}. Then $\mathcal{N}_{\lambda}^0=\{0\}$ for all $\lambda\in (0, \Lambda_0)$.
\end{lem}
\begin{proof}
We prove this lemma by contradiction. Let $w\in \mathcal{N}_{\lambda}^0$ such that $w\not\equiv 0$. Then from \eqref{fi2} we get

\begin{align}\nonumber
\|w\|^2 &=(p+q-1)\lambda\int_{-1}^1f(x)|w_1(x, 0)|^p|w_2(x,0)|^q  \\\label{fi2u}
&+\int_{-1}^1 \left\{(\alpha+\beta+2w_1^2+2w_2^2)(\alpha+\beta-1+2w_1^2+2w_2^2)+4(w_1^2+w_2^2)\right\}|w_1|^\alpha|w_2|^\beta e^{w_1^2+w_2^2} .
\end{align}
Using equation \eqref{nlam}, it is easy to show that $w\in \Gamma$. Now from equation \eqref{nlam} and \eqref{fi2u}, we get
\begin{align*}
(2-p-q)&\lambda\int_{-1}^1f(x)|w_1(x, 0)|^p|w_2(x,0)|^q dx\\&\geq \int_{-1}^1 (\alpha+\beta-p-q+2w_1^2+2w_2^2)(\alpha+\beta+2w_1^2+2w_2^2)|w_1|^\alpha|w_2|^\beta e^{w^2_1+w^2_2}dx\\
&{\geq \int_{-1}^1 (\alpha+\beta-2+2w_1^2+2w_2^2)(\alpha+\beta+2w_1^2+2w_2^2)|w_1|^\alpha|w_2|^\beta e^{w^2_1+w^2_2}dx}.
\end{align*}
Hence if $\lambda\in(0,\Lambda_0)$, we get a contradiction.
\end{proof}
\begin{lem}\label{cbb}
$I_{\lambda}$ is coercive and bounded below on $\mathcal{N}_{\lambda}$.
\end{lem}
\begin{proof}  Let $w\in \mathcal{N}_{\lambda}$. Using equations \eqref{nlam}, and the fact that $(\alpha+\beta)G(u,v)<g_1(u,v)u+g_2(u,v)v$, we get
\begin{eqnarray*}
I_\lambda(w)&=&\left(\frac{1}{2}-\frac{1}{\alpha+\beta}\right)\|w\|^2-\lambda\left(\frac{1}{p+q}-\frac{1}{\alpha+\beta}\right)\int_{-1}^1f(x)|w_1(x, 0)|^p|w_2(x,0)|^q dx \nonumber\\&-&\int_{-1}^1\left(G(w(x, 0))-\frac{1}{\alpha+\beta}(g_1(w(x, 0))w_1(x, 0)+g_2(w(x, 0))w_2(x, 0)\right)dx\\
&\geq&\left(\frac{1}{2}-\frac{1}{\alpha+\beta}\right)\|w\|^2-\lambda\left(\frac{1}{p+q}-\frac{1}{\alpha+\beta}\right)\int_{-1}^1f(x)|w_1(x, 0)|^p|w_2(x,0)|^q dx.\nonumber
\end{eqnarray*}
So using the fact that $\alpha+\beta>2>p+q$ and H\"{o}lder's inequality the lemma is proved.
\end{proof}
\noindent The following lemma shows that minimizers of $I_{\lambda}$ on $\mathcal{N}_{\lambda}$ are critical points of $I_{\lambda}$.
\begin{lem}
Let $w$ be a local minimizer of $I_{\lambda}$ in any decompositions of $\mathcal{N}_{\lambda}$ such that $w\not\in \mathcal{N}_{\lambda}^0$. Then w is a critical point of $I_{\lambda}$.
\end{lem}
\begin{proof}
If $w$ minimizes $I_{\lambda}$ in $\mathcal{N}_{\lambda}.$  Then by Lagrange multiplier theorem, we get
 \begin{equation}\label{lag}
 I_{\lambda}^{'}(w) = \nu {C}_{\lambda}^{'}(w ),\; \textrm{where}\; {C}_{\lambda}(u)=\langle {I}_{\lambda}^{'}({u}),{u}\rangle=0.
 \end{equation}
 Now
{\begin{equation*}
    \langle {I}_{\lambda}^{'}(w),w\rangle = \nu\langle {C}_{\lambda}^{'}(w),w\rangle = \nu \left(\Phi^{\prime\prime}_{w}(1)+ \langle {I}_{\lambda}^{'}(w),w\rangle\right)=0.
\end{equation*}}
But $\Phi^{''}_{w}(1) \neq 0$ as $w\not\in {N}_{\lambda}^{0}.$ Thus $\nu = 0.$  Hence by \eqref{lag} we get $I'_{\lambda}(w)=0.$
\end{proof}
\begin{lem}\label{zii}
Let $\Lambda_0$ be such that equation \eqref{gma} hold.
Then given $w\in \mathcal{N}_{\lambda} \setminus \{0\}$, there exist $\epsilon>0$
and a differentiable function $\xi : \textbf{B}(0,\epsilon)\subset \mathcal W^1_{0, L}(\mathcal C) \rightarrow \mathbb{R}$ such that $\xi(0)=1$, the function $\xi(v)(w-v)\in \mathcal{N} _{\lambda}$. Moreover, for all $v\in \mathcal W^1_{0, L}(\mathcal C), \langle \xi^{\prime}(0),v\rangle=\frac{A}{H}$, where $A$ and $H$ are defined as

{\begin{align*}
A=&I''_\lambda(w)(w,v)+\langle I'(w),v\rangle=\\
&\displaystyle 2\int_\mathcal C \nabla w_1\nabla v_1dxdy+\displaystyle 2\int_\mathcal C \nabla w_2\nabla v_2dxdy \\&-\displaystyle\frac{\lambda p^2}{p+q}\int_{-1}^1f(x)|w_1(x, 0)|^{p-2}w_1(x, 0)|w_2(x,0)|^qv_1(x,0)dx\\&-\displaystyle\frac{\lambda q^2}{p+q}\int_{-1}^1f(x)|w_1(x, 0)|^{p}|w_2(x,0)|^{q-2}w_2(x, 0)v_1(x,0)dx\\&-\displaystyle\int_{-1}^1\left\{(\alpha+2|w_1|^2+2|w_2|^2)(\alpha+2|w_1|^2)+4w_1^2\right\}|w_1|^{\alpha-2}v_1|w_2|^\beta e^{|w_1|^2+|w_2|^2}dx\\&-\displaystyle\int_{-1}^1\left\{(\beta+2|w_1|^2+2|w_2|^2)(\beta+2|w_2|^2)+4w_2^2\right\}|w_1|^\alpha|w_2|^{\beta-2} w_2v_2e^{|w_1|^2+|w_2|^2}dx,
\end{align*}}
{\begin{align*}
&H=I''_\lambda(w)(w,w)=(2-p-q)\|w\|^2\\
&-\displaystyle\int_{-1}^1\left\{(\alpha+\beta-p-q+2|w_1|^2+2|w_2|^2)(\alpha+\beta+2|w_1|^2+2|w_2|^2)+4(w_1^2+w_2^2)\right\}\\
&\hspace{10cm}\times|w_1|^\alpha|w_2|^\beta e^{|w_1|^2+|w_2|^2}dx.
\end{align*}}
\end{lem}
\begin{proof}
For fixed $w\in \mathcal{N}_{\lambda}\setminus \{0\}$, define $\mathcal{F}_w:\mathbb{R}\times \mathcal W^1_{0, L}(\mathcal C) \rightarrow \mathbb{R}$ as follows
\begin{eqnarray*}
\mathcal{F}_w(t,v) &=& \langle {I'_\lambda(t(w-v)),(w-v)}\rangle\\
 &=&t\|w-v\|^{2}-t^{p+q-1}\lambda \displaystyle\int_{-1}^1f(x)|w_1(x, 0)-v_1(x, 0)|^{p}|w_2(x, 0)-v_2(x, 0)|^{q}dx\\&-&
\displaystyle \int_{-1}^1g_1(t(w(x, 0)-v(x, 0)))(w_1(x, 0)-v_1(x, 0))dx\\&-&\displaystyle \int_{-1}^1g_2(t(w(x, 0)-v(x, 0)))(w_2(x, 0)-v_2(x, 0))dx.
\end{eqnarray*}
Then $\mathcal{F}_w(1,0) = 0,\; \frac{\partial}{ \partial t}\mathcal{F}_w(1,0)\neq 0$ as {$w\not\in \mathcal{N}_{\lambda}^{0}.$} So we can apply implicit function theorem to get a differentiable function $\xi : \mathcal{B}(0, \epsilon) \subseteq \mathcal W^1_{0, L}(\mathcal C) \rightarrow \mathbb{R}$ such that $\xi(0) = 1$ and $\langle \xi^{\prime}(0),v\rangle=\frac{A}{H}$ for the choice of $A$ and $H$ defined in the Lemma.
Moreover $\mathcal{F}_w(\xi(v),v) = 0$, $\textrm{for all}\; v \in \mathcal{B}(0, \epsilon)$
which implies $ \xi(v)(w-v) \in \mathcal{N}_{\lambda}$.
\end{proof}

\noindent We define $\theta_{\lambda}:=\inf\left\{I_{\lambda}(w) | w\in {\mathcal{N}_{\lambda}^+}\right\}$  and we prove the following lemma.
\begin{lem}\label{negl}
There exists a constant $C>0$ such that $\theta_{\lambda}<{-\left(\frac{(\alpha+\beta-p-q)(\alpha+\beta-2)}{2(p+q)}\right)C}$.
\end{lem}
\begin{proof}
Let $v\in \mathcal W^1_{0, L}(\mathcal C)$ be such that  $\int_{-1}^1 f(x)|v_1(x, 0)|^p|v_2(x,0)|^qdx>0$. Then from Lemma \ref{uhm}(ii), we get a $w\in \mathcal{N}_{\lambda}^+$.
\begin{eqnarray}\label{rhj}\nonumber
I_\lambda(w)=\left(\frac{p+q-2}{2(p+q)}\right)\|w\|^2&+&\frac{1}{p+q}\left(\int_{-1}^1 g_1(w(x,0))w_1(x, 0)dx\right.\\&+&\left.\int_{-1}^1 g_2(w(x,0))w_2(x, 0)dx\right)
-\int_{-1}^1 G(w(x, 0)) dx.
\end{eqnarray}
Also using \eqref{fi2}, we get
\begin{equation}\label{fi2in}
(2-p-q)\|w\|^2\geq\int_{-1}^1 (\alpha+\beta-p-q+2w_1^2+2w_2^2)(\alpha+\beta+2w_1^2+2w_2^2)|w_1|^\alpha|w_2|^\beta e^{w^2_1+w^2_2}dx.
\end{equation}
Now using \eqref{fi2in} in \eqref{rhj}, we get
\begin{align*}
& I_\lambda(w)\leq\\
&{\frac{-1}{2(p+q)}\int_{-1}^1 \left\{(\alpha+\beta-p-q+2w_1^2+2w_2^2)(\alpha+\beta-2+2w_1^2+2w_2^2)\right\}|w_1|^\alpha|w_2|^\beta e^{w^2_1+w^2_2}dx}\\
&{\leq}-{\frac{(\alpha+\beta-2)(\alpha+\beta-p-q)}{2(p+q)}\int_{-1}^1|w_1|^\alpha|w_2|^\beta e^{w^2_1+w^2_2} dx}\\
&{\leq}-{\frac{(\alpha+\beta-2)(\alpha+\beta-p-q)}{2(p+q)}C}
\end{align*}
where $C=\int_{-1}^1|w_1|^\alpha|w_2|^\beta $.
\end{proof}
In the next proposition we show the existence of a Palais-Smale sequence.
\begin{pro}\label{pscdp}
{There exists $0<\Lambda'_0\leq \Lambda_0$ such that for any $\lambda \in (0,\Lambda'_{0})$}, then there exists a minimizing sequence $\{w_k\}=(w_{k_1}, w_{k_2}) \subset \mathcal{N}_{\lambda}$ such that
\begin{equation}\label{pscd}
    {I}_{\lambda}(w_{k}) = \theta_{\lambda}+o_k(1) \;\textrm{and}\; {I}_{\lambda}^{'}(w_{k}) = o_k(1).
\end{equation}
\end{pro}
\begin{proof}
 From Lemma \ref{cbb}, $I_{\lambda}$ is bounded below on $\mathcal{N}_{\lambda}$. So by Ekeland variational principle, there exists a minimizing sequence $\{w_k\}\in \mathcal{N}_{\lambda}$ such that
\begin{eqnarray*}
I_{\lambda}(w_k)&\leq& \theta_{\lambda}+\frac{1}{k},\label{ek1}\\
I_{\lambda}(v)&\geq& I_{\lambda}(w_k)- \frac{1}{k}\|v-w_k\|\;\;\mbox{for all}\;\;v\in \mathcal{N}_{\lambda}.
\end{eqnarray*}
for all $v\in \mathcal{N}_\lambda$. Using equation \eqref{ek1} and Lemma \ref{negl}, it is easy to show that $w_k\neq 0$. From Lemma \ref{cbb},  we have that $\displaystyle\sup_{k}\Vert w_k\Vert<\infty$. Next we claim that $\|I^{\prime}_{\lambda}(w_k)\|\rightarrow 0$ as $k \rightarrow 0$.
Now, using the Lemma \ref{zii} we get the differentiable functions $\xi_k:\mathcal{B}(0, \epsilon_k)\rightarrow \mathbb{R}$ for some $\epsilon_k>0$ such that $\xi_k(v)(w_k-v)\in \mathcal{N}_{\lambda}$,\; $\textrm{for all}\;\; v\in \mathcal{B}(0, \epsilon_k).$
For fixed $k$, choose $0<\rho<\epsilon_k$. Let $w\in \mathcal W^1_{0, L}(\mathcal C)$ with $w\not\equiv 0$ and let $v_\rho=\frac{\rho w}{\|w\|}$. We set $\eta_\rho=\xi_k(v_\rho)(w_k-v_\rho)$. Since $\eta_\rho \in \mathcal{N}_{\lambda}$, we get from equation \eqref{nlam}
\begin{align*}
I_{\lambda}(\eta_\rho)-I_{\lambda}(w_k)\geq-\frac{1}{k}\|\eta_\rho-w_k\|.
\end{align*}
Now by mean value Theorem,
and using $\langle I_{\lambda}^{\prime}(\eta_\rho),(w_k-v_\rho)\rangle=0$, we get
\[
-\rho\langle I_{\lambda}^{\prime}(w_k),\frac{{w}}{\|{w}\|}\rangle+(\xi_k(v_\rho)-1)\langle I_{\lambda}^{\prime}(w_k)-I_{\lambda}^{\prime}(\eta_\rho),(w_k-v_\rho)\rangle \geq \frac{-1}{k}\|\eta_\rho-w_k\|+o_k(\|\eta_\rho-w_k\|).
\]
Thus
\begin{equation}\label{fifte}
\langle I_{\lambda}^{\prime}(w_k),\frac{w}{\|w\|}\rangle \leq \frac{1}{k\rho}\|\eta_\rho-w_k\|+\frac{o_k(\|\eta_\rho-w_k\|)}{\rho}+
\frac{(\xi_k(v_\rho)-1)}{\rho}\langle  I_{\lambda}^{\prime}(w_k)-I_{\lambda}^{\prime}(\eta_\rho),(w_k-v_\rho)\rangle.
\end{equation}
Since
$\displaystyle
\|\eta_\rho-w_k\|\leq \rho|\xi_k(v_\rho)|+|\xi_k(v_\rho)-1|\|w_k\|
$
and
$
\displaystyle\lim_{\rho\rightarrow 0^+}\frac{|\xi_k(v_\rho)-1|}{\rho}{=} \|\xi_k'(0)\|,$  taking limit  $\rho\rightarrow 0^+$\ in \eqref{fifte}, we get
\begin{equation*}
\langle I^{\prime}_{\lambda}(w_k),\frac{w}{\|w\|}\rangle\leq\frac{C}{k}(1+\|\xi_k^{'}(0)\|)
\end{equation*}
for some constant $C>0$, independent of $w$. So if we can show that $\|\xi_k^{'}(0)\|$ is bounded then we are done. {First, we observe that from the last inequality in the proof of Lemma \ref{cbb}, the boundednes of $\{w_k\}$ and Lemma 2.7, we obtain that $\Vert w_k\Vert\leq C\lambda$.
 Hence from Lemma \ref{zii}, taking $\Lambda_0$ small enough} and H\"older's inequality, for some $M>0$, we get
 $\langle \xi^{\prime}(0),v\rangle= \frac{M\|v\|}{H}$.
So to prove the claim we only need to prove that denominator $H$ in the expression of $\langle \xi^{\prime}(0),v\rangle$
 is bounded away from zero. Suppose not. Then there exists a subsequence still denoted $\{w_k\}$ such that
\begin{align}\nonumber
&(2-p-q)\|w_k\|^2-\displaystyle\int_{-1}^1(\alpha+\beta-p-q+2|w_{k_1}|^2+2|w_{k_2}|^2)(\alpha+\beta+2|w_{k_1}|^2+2|w_{k_2}|^2)\\\label{baw}
&\hspace{7cm}\times |w_{k_1}|^\alpha|w_{k_2}|^\beta e^{|w_{k_1}|^2+|w_{k_2}|^2}dx=o_k(1).
\end{align}
 Therefore $w_k\in \Gamma\setminus\{0\}$ for all $k$ large. Now using the fact that $ w_k\in \mathcal{N}_{\lambda}$, equation \eqref{baw}
 contradicts {\eqref{gma}} for $\lambda \in(0, \Lambda_0)$. Hence the proof of the lemma is now complete.
\end{proof}
\section{Existence and multiplicity results}
\setcounter{equation}{0}
\noindent In this section we show the existence and multiplicity of solutions by minimizing $I_{\lambda}$ on nonempty decomposition  of $\mathcal{N}_{\lambda}$ and generalized mountain pass theorem for suitable range of $\lambda$.
\begin{thm}
Let $\lambda$ be such that Lemma \ref{gmal} {and Proposition \ref{pscdp} hold}. Then there exists a function $\tilde{w} \in \mathcal{N}_{\lambda}^+$ such that $I_{\lambda}(\tilde{w})=\displaystyle \inf_{w\in {\mathcal{N}_{\lambda}^+}}I_{\lambda}(w)$.
\end{thm}
\begin{proof}
Using Lemma \ref{cbb} and {Proposition \ref{pscdp}}, we get a minimizing sequence $\{w_k\}$ in $\mathcal{N}_{\lambda}^+$ satisfying equation \eqref{pscd}. From Lemma \ref{cbb}, it follows that $\{w_k\}$ is bounded in $\mathcal W^1_{0, L}(\mathcal C)$. So up to subsequence,
$w_{k_1} \rightharpoonup \tilde{w}_1, w_{k_2} \rightharpoonup \tilde{w}_2 \;\textrm{in}\; \mathcal W^1_{0, L}(\mathcal C),
w_{k_1}(x,0)\rightarrow \tilde w_1(x,0), w_{k_2}(x,0)\rightarrow \tilde w_2(x,0)$ pointwise a.e. $(-1, 1)$ and
$w_{k_1}(x,0)\rightarrow \tilde w_1(x,0), w_{k_2}(x,0)\rightarrow \tilde w_2(x,0)$ in $L^r(-1,1)$ for all $ r>1.$
Now H\"older's inequality  implies that
 \begin{align*}\displaystyle \int_{-1}^1f(x)|w_{k_1}(x, 0)|^p|w_{k_2}(x,0)|^q dx\rightarrow \displaystyle \int_{-1}^1f(x)|\tilde{w}_1(x, 0)|^p|\tilde{w}_2(x,0)|^q dx.
 \end{align*}
 Also using compactness of $w\mapsto \int_{-1}^1f(x)|w_1(x, 0)|^p|w_2(x,0)|^q dx$  and the fact that $w_k\in \mathcal{N}_{\lambda}$, we get
{
\begin{equation}\label{vitali}
\int_{-1}^1g_1(w_k(x, 0)w_{k_1}(x, 0))dx<\infty \mbox{ and }\int_{-1}^1g_2(w_k(x, 0)w_{k_2}(x, 0))dx<\infty.
\end{equation}}
 So from the Vitali's convergence theorem,
 \begin{align*}
 &\int_{-1}^1g_1(w_{k_1}(x, 0),w_{k_2}(x, 0))\phi_1(x, 0)dx\rightarrow \int_{-1}^1g_1(\tilde{w}_1(x, 0),\tilde{w}_2(x, 0))\phi_1(x, 0)dx,\\
 &\int_{-1}^1g_2(w_{k_1}(x, 0),w_{k_2}(x, 0))\phi_2(x, 0)dx\rightarrow \int_{-1}^1g_2(\tilde{w}_1(x, 0),\tilde{w}_2(x, 0))\phi_2(x, 0)dx.
 \end{align*}
 Hence $\tilde{w}$ solves $(P_E)_{\lambda}$ and $\tilde{w}\in \mathcal{N}_{\lambda}$.
 Next we show that $\tilde{w}\in \mathcal{N}_{\lambda}^+$. Now using Lemma \ref{negl} and \eqref{vitali} , we get $\displaystyle \int_{-1}^1f(x)|\tilde{w}_1(x, 0)|^p|\tilde{w}_2(x,0)|^q dx>0$. Hence by Lemma \ref{uhm}(ii), there exists $t^+(\tilde{w})$ such that $t^+\tilde{w}\in \mathcal{N}_{\lambda}^+$. Our claim is that $t^+(\tilde{w})=1$. Suppose $t^+(\tilde{w})<1$ then $t^-(\tilde{w})=1$. So $\tilde{w}\in \mathcal{N}_{\lambda}^-$. Now from Lemma \ref{uhm} (ii), {$I_{\lambda}(t^+(\tilde{w})\tilde{w})< I_{\lambda}(\tilde w)\leq\theta_{\lambda}$}, which is impossible as $t^+(\tilde{w})\tilde{w}\in \mathcal{N}_{\lambda}^+.$ This completes the proof of the theorem.
\end{proof}
\begin{thm}
Let $\lambda$ be such that Lemma \ref{gmal} {and Proposition \ref{pscdp} hold}. Then, the function $\tilde{w} \in \mathcal{N}_{\lambda}^+$ is a local minimum of $I_{\lambda}(w)$ in $\mathcal W^1_{0, L}(\mathcal C)$.
\end{thm}
\begin{proof}
Since $\tilde{w}\in \mathcal{N}_{\lambda}^{+}$, we have $t^+(\tilde{w})=1
<t_*(\tilde{w})$. Hence { since $\Psi_w''(t)<0$ for $t>0$, $w\mapsto t_*(w)$ is continuous and  given}
$\epsilon>0$, there exists $\delta=\delta(\epsilon)>0$ such that $1+\epsilon< t_*(\tilde{w}-w)$
for all {$\|w\|<\delta$}. Also, from Lemma \ref{zii} we have, for $\delta>0$
small enough, we obtain a $C^1$ map $t: \textbf{B}(0,\delta)\rightarrow \mathbb R^+$
such that $t(w)(\tilde{w}-w)\in  \mathcal{N}_{\lambda}$, $t(0)=1$. Therefore, for
$\delta>0$ small enough we have $t^+(\tilde{w}-w)=
t(w)<1+\epsilon<t_*(\tilde{w}-w)$ for all $\|w\|<\delta$. Since $t_*(\tilde{w}-w)>1$,
we obtain $I_{\lambda}(\tilde{w})\leq {I_{\lambda}(t^+(\tilde{w}-w)(\tilde{w}-w))}\leq I_{\lambda}(\tilde{w}-w)$
for all $\|w\|<\delta$. This shows that $\tilde{w}$ is a local minimizer for
$I_{\lambda}$ in $\mathcal W^1_{0, L}(\mathcal C)$.\\
\end{proof}
\subsection{Existence of a second solution}\label{3.1}
\noindent Throughout this section, we fix $\lambda \in (0, \Lambda_0)$ and
  $w=(\tilde w_1,\tilde w_2)$ as the local minimum of $I_\lambda$
obtained in the previous section. Using min-max methods and Mountain pass lemma around a closed set, we prove
the existence of a second solution $(\tilde{z}_1 , \tilde{z}_2 )$
of $(P_E)_\lambda$ such that $\tilde{z}_1 \ge \tilde{w}_1 $ and
$\tilde{z}_2 \ge \tilde{w}_2 $ in $\mathcal C$.
\begin{defi}  \rm
Let $\mathcal{F} \subset\mathcal{W}^1_{0,L}(\mathcal C) $ be a closed set. We say that a sequence
$\{w_k\} \subset\mathcal{W}^1_{0,L}(\mathcal C) $ is a Palais-Smale sequence for $I_\lambda$ at
level $\rho$ around $\mathcal{F}$, and we denote $(PS)_{\mathcal{F},\rho}$,
if
\begin{gather*}
\lim_{k \to +\infty} dist \Big(w_k, \mathcal{F} \Big) = 0, \quad
\lim_{k \to +\infty} I_\lambda(w_k) = \rho, \quad
\lim_{k\to+\infty} \|I'_\lambda(w_k) \| = 0.
\end{gather*}
\end{defi}
\noindent We have the following version of compactness Lemma based on the Vitali's convergence theorem\cite{adi,fmr}:
\begin{lem}\label{strcri}
For any $(PS)_{F,\rho}$ sequence $\{w_k\}\subset \mathcal W^1_{0,L}(\mathcal C)$ of $I_\lambda$ for any closed set $\mathcal F$. Then there exists $w_0\in \mathcal W^{1}_{0,L}(\mathcal{C})$ such that, up to a subsequence,  $ g_1(w_k(x, 0)) \rightarrow   g_1(w_0(x, 0)) $ , \;$ g_2(w_k(x, 0)) \rightarrow   g_2(w_0(x, 0)) $ in $L^{1}(-1,1)$ and  $ G(w_k(x, 0))  \rightarrow  G(w_0(x, 0))$ in $L^1(-1,1)$.
\end{lem}
Let $T = \{ (z_1,z_2): z_1\ge \tilde{w}_1, z_2\ge \tilde{w}_2 \text{ a.e. in } \mathcal C \}$.
We note that $ \lim_{t\rightarrow+\infty} I(\tilde{w}_1 + t z_1, \tilde{w}_2 + t z_2) = -
\infty$ for any $(z_1, z_2) \in\mathcal{W}\setminus\{0\} $. Hence, we may fix
$(\overline{w}_1, \overline{w}_2) \in\mathcal{W}\setminus\{0\} $ such that
$I_\lambda(\tilde{w}_1+ \overline{w}_1, \tilde{w}_2 + \overline{w}_2 ) < 0 $.
We define the mountain pass level
\begin{equation}
\label{eq:13} \rho_0= \inf_{\gamma\in\Upsilon} \sup_{t \in[0,1] }
I_\lambda(\gamma(t)),
\end{equation}
where $ \Upsilon = \{ \gamma: [0,1] \to\mathcal{W}^1_{0,L}(\mathcal C) : {\gamma([0,1])\subset C}, \; \gamma(0)
= (\overline{w}_1, \overline{w}_2),
 { \gamma(1) = (\tilde{w}_1 + \overline{w}_1, \tilde{w}_2 + \overline{w}_2 ) } \}$.
  It follows that $\rho_0 \ge I_\lambda(\tilde{w}_1, \tilde{w}_2)$.
 If $\rho_0 = I_\lambda(\tilde{w}_1, \tilde{w}_2) $, we
obtain that $\inf\{I_\lambda(z_1,z_2): \|(z_1,z_2)-(\tilde{w}_1,\tilde{w}_2)\| = R \}
= I_\lambda(\tilde{w}_1, \tilde{w}_2)$
for all $R\in(0, R_0)$ for some $R_0$ small. We now let $\mathcal{F} = T $ if $\rho_0 > I_\lambda(\tilde{w}_1, \tilde{w}_2)$, and
$ \mathcal{F}= T \cap \{ \|(z_1,z_2)-(\tilde{w}_1,\tilde{w}_2)\|
= \frac{R_0}{2} \} $ if
$\rho_0 = I_\lambda(\tilde{w}_1, \tilde{w}_2) $.\\
\noindent Now we need the following version of the "sequence of Moser functions concentrated on the boundary" as defined in \cite{adyadav, jpsree} for the scalar case.
\begin{lem}\label{mser}
There exists a sequence $\{\phi_k\}\subset H^{1}_{0,L}(\mathcal{C})$ satisfying
\begin{enumerate}
\item $\phi_k\geq 0$, $\textrm{supp}(\phi_k)\subset B(0, 1)\cap \mathbb R^2_+$,
\item $\|\phi_k\|=1$,
\item $\phi_k$ is constant on $x\in B(0, \frac{1}{k})\cap \mathbb R^2_+$,  and $\phi_k^2=\frac{1}{\pi}\log k+O(1)$ for $x\in B(0, \frac{1}{k})\cap\mathbb R^2_+$.
\end{enumerate}
\end{lem}
\begin{proof}
Let \begin{equation*}
 \psi_{k}(x,y) = \frac{1}{\sqrt{2\pi}}\left\{
\begin{array}{lr}
\sqrt{\log k}& 0\leq \sqrt{x^2+y^2}\leq {\frac{1}{k}},\\
 \frac{\log {\frac{1}{\sqrt{x^2+y^2}}}}{\sqrt{\log k}} & \frac{1}{k}\leq \sqrt{x^2+y^2} \leq 1,\\
 0 & \sqrt{x^2+y^2}\geq 1.
\end{array}
\right.
\end{equation*}
Then $\displaystyle \int_{\mathbb R^2}|\nabla \psi_k|^2 dxdy=1$ and $\displaystyle \int_{\mathbb R^2}|\psi_k|^2 dxdy=O\left(\frac{1}{\log k}\right)$. Let $\displaystyle \tilde{\psi}_k=\psi_k|_{\mathbb R^2_+}$ and $\displaystyle \phi_k=\frac{\tilde \psi_k}{\|\tilde \psi_k\|}$. Then $\phi_k\geq 0$ and $\|\phi_k\|=1$. Also $\displaystyle \int_{\mathbb R^2_+}|\nabla \tilde \psi_k|^2 dx dy=\frac{1}{2}$ and $\displaystyle \int_{\mathbb R^2_+}|\tilde \psi_k|^2 dx dy=O\left(\frac{1}{\log k}\right)$. Therefore $\displaystyle\phi_k^2=\frac{1}{\pi}\log k+O(1)$.
\end{proof}

We have the following upper bound on $\rho_0$.
\begin{lemma}\label{lm-rho_0}
With $\rho_0$ defined as in \eqref{eq:13}, we have
$\rho_0 < I_\lambda(\tilde{w}_1, \tilde{w}_2) + \frac{\pi}{2}$.
\end{lemma}
\begin{proof}
\noindent We prove this lemma by contradiction. Define $\overline{\psi}_k(x)=\frac{1}{\sqrt 2}\phi_k(x)$, where $\phi_k$ are defined in Lemma \ref{mser} then we have the following
\begin{enumerate}
\item $\|(\bar \psi_k,\bar \psi_k)\|^2=1$,
\item $\displaystyle\int_{\mathbb R^2_+}|\bar \psi_k|^2 dxdy=O(\frac{1}{\log k})$,
\item $ \bar \psi_k^2=\frac{1}{2\pi}\log k+O(1)$.
  \end{enumerate}
  Now, suppose $\rho_0 \ge I_\lambda(\tilde{w}_1, \tilde{w}_2)+  \frac{\pi}{2}$ and we
derive a contradiction. This means that for some $t_k,\; s_k >0$:
$$
I_\lambda(\tilde{w}_1+t_k \psi_k, \tilde{w}_2+s_k \psi_k)
= \sup_{t, s >0} I_\lambda( \tilde{w}_1+ t \psi_k, \tilde{w}_2 + s \psi_k)
\ge I_\lambda(\tilde{w}_1, \tilde{w}_1)+ \frac{ \pi}{2}, \quad \forall k.
$$
Since {$I_\lambda(\tilde{w}_1+t z_1, \tilde{w}_1+s z_2) \to -\infty$} as
$t,s \to +\infty$, we obtain that $(t_k, s_k)$ is bounded in $\mathbb{R}^2$.
Then, using {$\| (\bar\psi_k, \bar\psi_k)\| = 1$}, we obtain
{\begin{align*}
\frac{ t_k^2 + s_k^{2} }{4} &+ t_k\int_{-1}^1\nabla \tilde w_1\nabla \bar\psi_kdx+s_k\int_{-1}^1\nabla \tilde w_2\nabla\bar\psi_kdx \\
&\geq \frac{\lambda}{p+q}\left(\int_{-1}^1f(x)|\tilde w_1+t_k\bar\psi_k|^p|\tilde w_2+s_k\bar\psi_k|^q dx-\int_{-1}^1f(x)|\tilde w_1|^p|\tilde w_2|^q dx\right)\\
&+\int_{-1}^1G(\tilde w_1+t_k\bar\psi_k, \tilde w_2+s_k\bar\psi_k) dx-\int_{-1}^1 G(\tilde w_1,\tilde w_2) dx+\frac{\pi}{2}
\end{align*}}
{from which {together with the fact that $(\tilde w_1,\tilde w_2)$ is a solution for problem $(P_E)_\lambda$} we obtain
\begin{align*}
\frac{ t_k^2 + s_k^{2} }{4} &+\frac{\lambda p}{p+q} t_k\int_{-1}^1f(x)\vert\tilde w_1\vert^{p-2}\tilde w_1\bar\psi_k \vert \tilde w_2\vert^qdx
+\frac{\lambda q}{p+q} s_k\int_{-1}^1f(x)\vert\tilde w_1\vert^{p}\vert\tilde w_2\vert^{q-2}\tilde w_2\bar\psi_k dx\\
&\geq \frac{\lambda}{p+q}\left(\int_{-1}^1f(x)|\tilde w_1+t_k\bar\psi_k|^p|\tilde w_2+s_k\bar\psi_k|^q dx-\int_{-1}^1f(x)|\tilde w_1|^p|\tilde w_2|^q dx\right)\\
&+\int_{-1}^1G(\tilde w_1+t_k\bar\psi_k, \tilde w_2+s_k\bar\psi_k) dx-\int_{-1}^1 G(\tilde w_1,\tilde w_2) dx\\
&-t_k\int_{-1}^1g_1(\tilde w_1, \tilde w_2)\bar\psi_k dx-s_k\int_{-1}^1g_2(\tilde w_1, \tilde w_2)\bar\psi_k dx+\frac{\pi}{2}.
\end{align*}}
Now using that {$G$ is convex}, we get
{
\begin{align*}
\frac{ t_k^2 + s_k^{2} }{4} &+\frac{\lambda p}{p+q} t_k\int_{-1}^1f(x)\vert\tilde w_1\vert^{p-2}\tilde w_1\bar\psi_k \vert \tilde w_2\vert^qdx
+\frac{\lambda q}{p+q} s_k\int_{-1}^1f(x)\vert\tilde w_1\vert^{p}\vert\tilde w_2\vert^{q-2}\tilde w_2\bar\psi_k dx\\
&-\frac{\lambda}{p+q}\left(\int_{-1}^1f(x)|\tilde w_1+t_k\bar\psi_k|^p|\tilde w_2+s_k\bar\psi_k|^q dx-\int_{-1}^1f(x)|\tilde w_1|^p|\tilde w_2|^q dx\right)\geq \frac{\pi}{2}.
\end{align*}}
{ From the Taylor expansion, we have for some $\theta_1$, $\theta_2\in (0,1)$
\begin{align*}
&\frac{\lambda p}{p+q} t_k\int_{-1}^1f(x)\vert\tilde w_1\vert^{p-2}\tilde w_1\bar\psi_k \vert \tilde w_2\vert^qdx
+\frac{\lambda q}{p+q} s_k\int_{-1}^1f(x)\vert\tilde w_1\vert^{p}\vert\tilde w_2\vert^{q-2}\tilde w_2\bar\psi_k dx\\
&-\frac{\lambda}{p+q}\left(\int_{-1}^1f(x)|\tilde w_1+t_k\bar\psi_k|^p|\tilde w_2+s_k\bar\psi_k|^q dx-\int_{-1}^1f(x)|\tilde w_1|^p|\tilde w_2|^q dx\right)\\
&=\lambda\frac{p(p-1)}{p+q}t_k^2\int_1^1f(x)\vert\tilde w_1+\theta_1t_k\bar\psi\vert^{p-2}\vert \tilde w_2\vert^q\bar\psi_k^2 dx\\
&+\lambda\frac{q(q-1)}{p+q}s_k^2\int_1^1f(x)\vert\tilde w_1\vert^{p}\vert \tilde w_2\vert^q\vert \tilde w_2+\theta_2\bar\psi_k\vert^{q-2}\bar\psi_k^2 dx=(t_k^2+s_k^2)O\left(\frac{1}{\log(k)}\right).
\end{align*}}
{Hence we obtain that
\begin{equation*}
t_k^2+s_k^2\geq 2\pi(1-O\left(\frac{1}{\log(k)}\right)).
\end{equation*}}
\noindent Since $(t_k, s_k)$ is the critical point for $I_\lambda( \tilde{w}_1+ t \bar\psi_k, \tilde{w}_2 + s \bar\psi_k)$, we have
\begin{equation*}
{\frac{d}{dt}I_\lambda( \tilde{w}_1+ t \bar\psi_k, \tilde{w}_2 + s \bar\psi_k)|_{(t,s)=(t_k,s_k)}=0}, {\frac{d}{ds}I_\lambda( \tilde{w}_1+ t \bar\psi_k, \tilde{w}_2 + s \bar\psi_k)|_{(t,s)=(t_k,s_k)}=0}.
\end{equation*}
 Therefore
{\begin{align*}
\frac{ t_k^2 + s_k^2 }{2} &+ t_k\int_{-1}^1\nabla \tilde w_1\nabla\bar \psi_kdx+s_k\int_{-1}^1\nabla \tilde w_2\nabla \bar\psi_kdx \\
&= \frac{\lambda p }{p+q}\int_{-1}^1f(x)|\tilde w_1+t_k\bar\psi_k|^{p-2}(\tilde w_1+t_k\bar\psi_k)|\tilde w_2+s_k\bar\psi_k|^q t_k\bar\psi_kdx\\
&+\frac{\lambda q }{p+q}\int_{-1}^1f(x)|\tilde w_1+t_k\bar\psi_k|^{p}|\tilde w_2+s_k\bar\psi_k|^{q-2}(\tilde w_2+s_k\bar\psi_k) s_k\bar\psi_kdx\\
&+\int_{-1}^1g_1(\tilde w_1+t_k\bar\psi_k, \tilde w_2+s_k\bar\psi_k)t_k\bar\psi_k dx+\int_{-1}^1g_2(\tilde w_1+t_k\bar\psi_k, \tilde w_2+s_k\bar\psi_k)s_k\bar\psi_k dx.
\end{align*}}
{ Without loss of generality, we can assume that $f$ is positive in the neighborhood of $0$, then we get
\begin{align*}\nonumber
t_k^2+s_k^2+ t_k\int_{-1}^1\nabla \tilde w_1\nabla\bar \psi_kdx+s_k\int_{-1}^1\nabla \tilde w_2\nabla\bar \psi_kdx
&\geq C\int_{-\frac{1}{k}}^{\frac{1}{k}} e^{t_k^2\psi_k^2+s_k^2\bar \psi_k^2}(t_k+s_k)\bar \psi_kdx\\
&\geq \frac{C}{\sqrt \pi}e^{\left(\frac{t_k^2+s_k^2}{2\pi}-1\right)}(t_k+s_k)\sqrt {\log k}.\nonumber
\end{align*}}
This and {the fact that $t_k^2+s_k^2\geq 2\pi(1-O\left(\frac{1}{\log(k)}\right)) $ imply that
 $ t_k^2 + s_k^2 \to \infty $ as
$k \to \infty$, which is a contradiction}.
\end{proof}
\noindent We use the following version of {Lions'} higher integrability Lemma. The proof follows the same steps as of Lemma 3.4 in \cite{nhsr}.
\begin{lemma}\label{lm-lions}
 Let $\{w_k\}$ be a sequence in $\mathcal{W}^1_{0,L}(\mathcal C)$ such that
$\|w_k\|=1$, for all $k$ and
$w_{k_1} \rightharpoonup w_1$ in ${H}^1_{0,L}(\mathcal C)$ and $w_{k_2} \rightharpoonup w_2$ in ${H}^1_{0,L}(\mathcal C)$
for some $(w_1,w_2)\ne (0,0)$ in $\mathcal{W}^1_{0,L}(\mathcal C)$. Then, for
$ 0<p<  \pi (1-\|(w_1,w_2)\|^2)^{-1}$,
\[
\sup_{k\geq 1} \int_{-1}^1 e^{p (w_{k_1}^2 +w_{k_2}^2)} dx < \infty.
 \]
\end{lemma}
%
\noindent Now, we prove our main result.
\begin{thm}
For $\lambda\in(0,\Lambda_0)$, problem $(P_E)_\lambda$ has a second
nontrivial solution $\hat w=(\hat{w}_1 , \hat{w}_2 )$ such that
$\hat{w}_1\ge \tilde{w}_1 > 0 $ and $\hat{w}_2 \ge \tilde{w}_2 > 0 $
in $\mathcal C$.
\end{thm}
\begin{proof}
Let $ \{w_k\} $ be a Palais-Smale sequence for $I_\lambda$ at the level
$\rho_0$ around $\mathcal{F} $ obtained by applying Ekeland Variational principle on $\mathcal{F}$. Then it is easy to verify that the sequence $\{w_k\}$ is bounded in $\mathcal{W}^1_{0,L}(\mathcal C)$. So upto a subsequence
$w_{k_1}\rightharpoonup \hat w_1$ and $ w_{k_2}\rightharpoonup \hat w_2$ in $H^1_{0,L}(\mathcal C)$.
It can be shown that $(\hat w_1, \hat w_2)$ is
a solution of $(P_E)_\lambda$ . It remains to show that
$(\hat w_1, \hat w_2) \not\equiv (\tilde w_1, \tilde w_2)$. We prove it by contradiction. Suppose it is not so. Then we have only the two following cases:\\
\noindent\textbf{Case 1:} $\rho_0 = I_\lambda(\tilde w_1, \tilde w_2)$.\\
In this case, using Lemma \ref{strcri} we have
\begin{align*}
 I_\lambda(\tilde w_1, \tilde w_2)+ o(1) &= I_\lambda(w_{k_1},w_{k_2})\\
&= \frac{1}{2}\|w_k\|^2-\frac{\lambda}{p+q}\int_{-1}^1f(x)|w_{k_1}|^p|w_{k_2}|^q dx
- \int_{-1}^1 G (w_{k_1}, w_{k_2}) dx .
\end{align*}
Thus from Lemma \ref{strcri}, we
have  $\| (w_k -\tilde w)\|= o(1)$, which
contradicts the fact that $w_k \in \mathcal{F} $.
\noindent\textbf{Case 2:} $\rho_0 \neq I_\lambda(\tilde w_1, \tilde w_2) $.\\
\noindent  In this case $\rho_0 -I_\lambda(\tilde w_1, \tilde w_2) \in (0, \frac{\pi}{2})$ and $I_\lambda(w_{k_1}, w_{k_2}) \to
\rho_0$.\\
Let $\beta_0 = \frac{\lambda}{p+q}\int_{-1}^{1}f(x)|w_{k_1}(x,0)|^p|w_{k_2}(x,0)|^qdx+\int_{-1}^1 G(\tilde w_1, \tilde w_2) dx$. Then from Lemma \ref{strcri},
\begin{equation}
\label{eq:16}
\frac{1}{2}\|(w_{k_1}, w_{k_2})\|^2 \to (\rho_0+\beta_0) \quad
\text{as } k \to \infty.
\end{equation}
By Lemma \ref{lm-rho_0}, for $\epsilon>0$ small, we have
$$
(1+\epsilon) (\rho_0 -I_\lambda(\tilde w_1,\tilde w_2))) < \frac{\pi}{2}.
$$
Hence, from \eqref{eq:16} we have
\begin{align*}
(1+\epsilon) \|(w_{k_1}, w_{k_2})\|^{2}
&< \pi \frac{\rho_0+\beta_0}{\rho_0+\beta_0 -\frac{1}{2}\|(\tilde w_1,\tilde w_2)\|^{2}}\\
&< \pi \Big( 1-\frac{1}{2} ( \frac{\|(\tilde w_1,\tilde w_2)\|^{2}}{\rho_0+\beta_0})
 \Big)^{-1}\\
&< \pi \Big( 1 - \| \frac{w_{k_1}}{\sqrt{2(\rho_0+\beta_0)}} \|^2
- \| \frac{w_{k_2}}{\sqrt{2(\rho_0+\beta_0)}}\|^2 \Big)^{-1}.
\end{align*}
Now, choose  $p $ such that
$$
(1+\epsilon) \|(w_{k_1}, w_{k_2})\|^{2} \le p
< \pi ( 1 - \| \frac{w_{k_1}}{\sqrt{2(\rho_0+\beta_0)}} \|^2
- \| \frac{w_{k_2}}{\sqrt{2(\rho_0+\beta_0)}}\|^2)^{-1}.
$$
Since
$\frac{w_{k_1}}{\|(w_{k_1},w_{k_2})\|} \rightharpoonup
 \frac{\tilde w_1}{\sqrt{2 (\rho_0 +\beta_0)}}$
and
$\frac{w_{k_2}}{\|(w_{k_1}, w_{k_2})\|} \rightharpoonup
\frac{\tilde w_2}{\sqrt{2 (\rho_0 +\beta_0)}}$ weakly in $H^1_{0, L}(-1, 1)$,
by Lemma \ref{lm-lions}, we have
\begin{equation}\label{eq:new}
\sup_k \int_{\Omega} \exp\Big( p
\big[ \big( \frac{w_{k_1}}{\|(w_{k_1}, w_{k_2})\|} \big)^{2} + \big( \frac{w_{k_2}}{\|(w_{k_1}, w_{k_2}) \|} \big)^{2} \big] \Big) dx
 < \infty.
\end{equation}
Now, from \eqref{eq:new},
together with  H\"older inequality we have as $k\to\infty$
$$
\int_{\Omega} g_1(w_{k_1},w_{k_2}) w_{k_1} dx \rightarrow
\int_{\Omega} g_1 (\tilde w_1, \tilde w_2) \tilde w_1dx;\quad \int_{\Omega} g_2(w_{k_1},w_{k_2}) w_{k_2} dx \rightarrow
\int_{\Omega} g_2 (\tilde w_1, \tilde w_2) \tilde w_2dx .
$$
{Hence,
\begin{align*}
o(1)&= I_\lambda' (w_{k_1},w_{k_2})(w_{k_1},0)\\
&= \int_{-1}^1 |\nabla w_{k_1}|^2 dx-\lambda \frac{ p}{p+q}\int_{-1}^1f(x)\vert w_{k_1}\vert^p\vert w_{k_2}\vert^q dx
-\int_{-1}^1 g_1(w_{k_1},w_{k_2}) w_{k_1} dx\\
&=  I_\lambda' (\tilde w_{1},\tilde w_{2})(\tilde w_{1},0)+ o(1)
\end{align*}
 from which we deduce $w_{k_1}\to \tilde w_1$. Similarly, we obtain $w_{k_2}\to \tilde w_2$.}
This is a contradiction to the assumption that $\rho_0 \ne { I_\lambda(\tilde w_1, \tilde w_2) } $.
\end{proof}
\section{Superlinear problems}
\noindent In this section we consider existence of solution for the problem $(P)$
We use the idea of harmonic extension to solve the problem $(P)$. The extension problem corresponding to the problem $(P)$ can be considered as
$$ \quad (P_E)\; \left\{
\begin{array}{rrll}
 \quad  -\Delta w_1=-\Delta w_2 &= &0, \;w_1>0,w_2>0\; \; \text{in}\;\mathcal{C}=(-1,1)\times (0,\infty),\\
  w_1=w_2 &=&0 \; \text{on}\; \{-1,1\}\times (0,\infty),\\
  \frac{\partial w_1}{\partial \nu}&=&h_1(w_1,w_2) \; \text{on} \; (-1,1)\times \{0\},\\
  \frac{\partial w_2}{\partial \nu}&=&h_2(w_1,w_2) \; \text{on} \; (-1,1)\times \{0\}.\\
\end{array}
\right.
$$
where $\frac{\partial w_1}{\partial \nu}=\displaystyle\lim_{y\rightarrow 0^+}\frac{\partial w_1}{\partial y}(x, y)$ and $\frac{\partial w_2}{\partial \nu}=\displaystyle\lim_{y\rightarrow 0^+}\frac{\partial w_2}{\partial y}(x, y)$.
As discussed in the introduction, the problem in $(P)$ is equivalent to solving $(P_E)$ on $\mathcal W^{1}_{0,L}(\mathcal{C})$.
The variational functional, $I:\mathcal W^{1}_{0,L}(\mathcal{C})\rightarrow \mathbb R$ related to the problem $(P_E)$ is given as
\begin{equation*}
I(w_1,w_2)=\frac{1}{2}\int_{\mathcal{C}}|\nabla w_1|^2 dxdy+\frac{1}{2}\int_{\mathcal{C}}|\nabla w_2|^2 dxdy-\int_{-1}^{1} H(w_1(x,0),w_2(x,0))dx.
\end{equation*}
Any function $w\in \mathcal W^{1}_{0,L}(\mathcal{C})$ is called the weak solution of $(P_E)$ if for any $\phi=(\phi_1, \phi_2)\in \mathcal W^{1}_{0,L}(\mathcal{C})$
\begin{equation}\label{solp}
\int_{\mathcal{C}} \left(\nabla w_1.\nabla \phi_1 + \nabla w_2.\nabla \phi_2\right) dxdy=\int_{-1}^{1}\left( h_1(w(x,0))\phi_1(x,0)+ h_2(w(x,0))\phi_2(x,0)\right) dx
\end{equation}
It is clear that critical points of $I$ in $\mathcal W^{1}_{0,L}(\mathcal{C})$ corresponds to the critical points of $J$ in $ H^{\frac{1}{2}}_0(\mathbb {R})$.
Thus if $(w_1,w_2)$ solves $(P_E)$ then $(u,v)= \textrm{trace}(w_1,w_2)=(w_1(x, 0),w_2(x,0))$ is the solution of problem $(P)$ and vice versa.\\
%
%
\noindent We will use { the mountain pass lemma} to show the existence of a solution in the critical case.
\begin{lem}
Assume that the conditions $(h1)-(h5)$ hold. Then $I$ satisfies  the mountain pass geometry around $0$.
\end{lem}
\begin{proof}
Using {assumption $(h_4)$} , we get
\begin{equation*}
H(s,t)\geq C_1 |s|^{\mu_1}+C_2|t|^{\mu_2}-C_3
\end{equation*}
for some $C_1, C_2, C_3>0$ and $\mu_1, \mu_2>2$. Hence for function $w\in \mathcal W^{1}_{0,L}(\mathcal{C})\cap C^\infty$ with support in $[-1,1]\times (0,1)$, we get
\begin{equation*}
I(tw)\leq \frac{1}{2}t^2\|w\|^2-C_1 t^{\mu_1} \int_{-1}^{1} |w_1(x, 0)|^{\mu_1} dx-C_2 t^{\mu_2} \int_{-1}^{1} |w_2(x, 0)|^{\mu_2} dx+C_3.
\end{equation*}
Hence $I(tw)\rightarrow -\infty$ as $t\rightarrow \infty$.
Next we will show that there exists $\alpha, \rho>0$ such that $I(w)>\alpha$ for all $\|w\|<\rho$. From $(h1)$ and $(h5)$,
for $\epsilon>0$ {$r>2$} there exists $C_1>0$ such that
{\begin{equation*}
|H(s,t)|\leq \frac{\lambda_1-\epsilon}{2}(s^2+t^2)+C_1(|s|^{r}+|t|^{r})e^{(1+\epsilon)(s^2+t^2)}.
\end{equation*}}
Hence, using the H\"{o}lder's inequality, {\eqref{eq12}} {and for $\rho$ small enough}, we get  for {$\ell>1$}
{\begin{equation*}
\int_{-1}^{1}|H(w(x,0))|dx\leq \frac{\lambda_1-\epsilon}{2}\|w(x,0)\|^2_{L^2(-1,1)}+C_1\|w_1(x,0)\|^{r}_{L^{r\ell}(-1,1)}+C_2\|w_2(x,0)\|^{r}_{L^{r\ell}(-1,1)}.
\end{equation*}}
Now using {Sobolev  embedding} and choosing $\|w\|= \rho$ for sufficiently small $\rho$, we get
\begin{equation*}
I(w) \geq \frac{1}{2}\left(1-\frac{\lambda_1-\epsilon}{\lambda_1} \right)\rho^2-C_3\rho^r.
\end{equation*}
Hence we get $\alpha>0$ such that $I(w)>\alpha$ for all $\|w\|=\rho$ for sufficiently small $\rho$.
\end{proof}
\noindent Next we show the boundedness of Palais-Smale sequences {($(PS)$ sequences for short)}.
\begin{lem}\label{psbo}
Every Palais-Smale sequence of $I$ is bounded in $H^{1}_{0,L}(\mathcal{C})$.
\end{lem}
\begin{proof}
Let $\{w_k\}=\{(w_{k_1},w_{k_2})\}$ be a $(PS)_c$ sequence, {that is}
\begin{align}
&\frac{1}{2}\|w_k\|^2-\int_{-1}^{1} H(w_k(x, 0))dx=c+o(1)\; \text{and} \label{plsm}\\
&\|w_k\|^2-\int_{-1}^{1} h_1(w_k(x, 0))w_{k_1}(x, 0)dx-\int_{-1}^{1} h_2(w_k(x, 0))w_{k_2}(x, 0)dx={o(\|w_k\|)}\label{plsm1}.
\end{align}
Therefore,
\begin{equation*}
\left(\frac{1}{2}-\frac{1}{\mu}\right)\|w_k\|^2-\frac{1}{\mu}\int_{-1}^{1} \left(\mu H(w_k)-h_1(w_k)w_{k_1}-h_2(w_k)w_{k_2}\right) dx=c+{o(\|w_k\|)}.
\end{equation*}
Using assumption $(h2)$, we get $\|w_k\|\leq C$ for some $C>0$.
\end{proof}
\noindent We have the following version of compactness Lemma:
\begin{lem}\label{lem2.4}
For any $(PS)_c$ sequence $\{w_k\}$ of $I$, there exists $\ddot{w} \in \mathcal W^{1}_{0,L}(\mathcal{C})$ such that, up to subsequence,  $ h_1(w_k(x, 0)) \rightarrow   h_1(\ddot{w}(x, 0)) $ in $L^{1}(-1,1)$, $ h_2(w_k(x, 0)) \rightarrow   h_2(\ddot{w}(x, 0)) $ in $L^{1}(-1,1)$ and  $ H(w_k(x, 0))  \rightarrow  H(\ddot{w}(x, 0))$ in $L^1(-1,1)$.
\end{lem}
\noindent Define $\Gamma=\{\gamma\in C([0, 1]; \mathcal W^{1}_{0,L}(\mathcal{C})):\gamma(0)=0 \;\textrm{and}\;I(\gamma(1))<0\}$ and the corresponding mountain pass level as $c=\displaystyle \inf_{\gamma \in \Gamma } \displaystyle \max_{t\in[0, 1]}I(\gamma(t))$.
\begin{lem}\label{lem2.5}
$c<\frac{\pi}{2}$.
\end{lem}
\begin{proof}
We prove by contradiction. Suppose $c\geq {\pi}{2}$.  Then we have
\begin{equation}\label{contra}
\displaystyle \sup_{t,s\geq 0}I(s\bar \psi_k,t\bar\psi_k)=I(s_k\bar\psi_k, t_k\bar\psi_k)\geq \frac{\pi}{2}
\end{equation}
where functions $\phi_k$ are given by Lemma \ref{lm-rho_0}. From equation \eqref{contra} and since {$H\geq 0$} on $\mathbb R^2_+$, we get
\begin{equation*}
s_k^2+t_k^2\geq 2\pi.
\end{equation*}
 Now as $(s_k, t_k)$ is point of maximum we get $\frac{\partial}{\partial s}I(s\bar\psi_k, t\bar\psi_k)|_{s=s_k}=0$ and $\frac{\partial}{\partial t}I(s\bar\psi_k, t\bar\psi_k)|_{t=t_k}=0$. Therefore by $(h4)$,
{\begin{align*}
\frac{s_k^2}{2}+\frac{t_k^2}{2}&=s_k^2\|\bar\psi_k\|^2+t_k^2\|\bar\psi_k\|^2\\
&=\int_{-1}^{1} \left( h_1(s_k\bar\psi_k(x,0), t_k\bar\psi_k(x, 0))s_k +  h_2(s_k\bar\psi_k(x,0), t_k\bar\psi_k(x, 0))t_k \right)\bar\psi_k(x, 0)dx\\
&\geq 2\sqrt{\log(k)}(s_k h_1(\frac{s_k\sqrt{\log(k)}}{\sqrt{2\pi}},\frac{t_k\sqrt{\log(k)}}{\sqrt{2\pi}}) +t_k h_2(\frac{s_k\sqrt{\log(k)}}{\sqrt{2\pi}},\frac{t_k\sqrt{\log(k)}}{\sqrt{2\pi}})e^{-\log(k)}\\
&\to\infty \quad \text{as } k\to\infty
\end{align*}}
{which contradicts the boundedness of $(t_k)_{k\in\N}$ and $(s_k)_{k\in \N}$.}
\end{proof}
\noindent Next we prove Theorem \ref{mht0} using the above Lemmas.\\
\noindent \textbf{Proof of Theorem \ref{mht0}:} Using {the variational Ekeland principle} and Lemma \ref{psbo}, there exists a bounded $(PS)_c$ sequence. So there exists $\ddot {w}\in \mathcal W^{1}_{0,L}(\mathcal{C})$ such that, upto a subsequence, $w_k\rightharpoonup \ddot {w}$ in $\mathcal W^{1}_{0,L}(\mathcal{C})$ and $w_k(x, 0)\rightarrow \ddot {w}(x, 0)$ pointwise. We first prove that $\ddot w$ solves the problem, then we show that $\ddot {w}$ is non zero. From Lemma \ref{psbo} and equation \eqref{plsm} {together with $h_1$, $h_2\geq 0$}, we get for some constant $C>0$,
\begin{equation*}
\int_{-1}^{1} h_1(w_k(x, 0))w_{k_1}(x, 0)dx{\leq}C,\;\int_{-1}^{1} h_2(w_k(x, 0))w_{k_2}(x, 0)dx{\leq}C, \;\int_{-1}^{1} H(w_k(x, 0))dx{\leq}C.
\end{equation*}
Now from Lemma \ref{lem2.4}, we get $h_1(w_k(x,0))\rightarrow h_1(\ddot {w}(x,0))$ in $L^1(-1,1)$. So for $\psi\in C_c^\infty$ the equation \eqref{solp} holds. Hence from density of $C_c^\infty$ in $\mathcal W^{1}_{0,L}(\mathcal{C})$,  $\ddot w$ is weak solution of $(P_E)$.\\
\noindent Next we claim that $\ddot w\not\equiv 0$. Suppose not.
Then from Lemma \ref{lem2.4}, we get $H(w_k(x,0))\rightarrow 0$ in $L^1({\mathbb R})$. Hence from equations \eqref{plsm} and \eqref{plsm1},
we get $\frac{1}{2}\|w_k\|^2\rightarrow c$ as $k\rightarrow \infty$. {Hence, from Lemma \ref{lem2.5} $\sup_k\|w_k\|^2\leq \pi-\epsilon$} for some $\epsilon>0$. Let $0<\delta <\frac{\epsilon}{\pi}$ and {$q=\frac{\pi}{(1+\delta)(1+\epsilon)(\pi-\epsilon)}>1$}. Using
Moser- Trudinger inequality \eqref{mtsys}, we get
\begin{eqnarray*}
\int_{-1}^{1}| h_1( w_k)w_{k_1}|^q  dx &\leq& A\int_{-1}^1|w_{k_1}(x, 0)|^{2q}dx+C\int_{-1}^1{e^{q(1+\epsilon)(1+\delta)(w_{k_1}^{2}+w_{k_2}^2)}}dx\\
&\leq& C_1\|w_k\|^{2q}+ C\int_{-1}^{1} e^{q(1+\delta){(1+\epsilon)}\|w_k\|^2\frac{w_{k_1}^2+w_{k_2}^2}{\|w_k\|^2}} dx <\infty.
\end{eqnarray*}
Similarly $\int_{-1}^{1}| h_1( w_k)w_{k_1}|^q  dx <\infty.$ Therefore by
\begin{equation*}
\int_{-1}^1 h_1(w_k(x, 0)) w_{k_1}(x, 0)  dx \rightarrow 0, \;\int_{-1}^1 h_2(w_k(x, 0)) w_{k_2}(x, 0)  dx \rightarrow 0
\end{equation*}
  and from equation \eqref{plsm}, we get $\displaystyle \lim_{k}\|w_k\|^2=0$, which {contradicts $c\geq \alpha $}. Hence $\ddot w$ is a nontrivial solution of the problem $(P_E)$.\\
  
\noindent {\bf Acknowledgements:} the authors were funded by IFCAM (Indo-French Centre for Applied Mathematics) UMI CNRS under the project "Singular phenomena in reaction diffusion equations and in conservation laws".

%
%

 \end{document}